\documentclass[12pt]{amsart}
\usepackage{fullpage}

\usepackage{setspace}
\RequirePackage{amssymb}
\usepackage{amsbsy}
\usepackage{enumerate}
\usepackage{mathrsfs}
\usepackage{undertilde}
\usepackage{amsmath}
\usepackage{url}

\theoremstyle{plain}
\newtheorem{thm}{Theorem}[section]
\newtheorem*{thm*}{Theorem}
\newtheorem*{mainthm*}{Main Theorem}
\newtheorem{lem}[thm]{Lemma} \newtheorem*{lem*}{Lemma}
 \newtheorem*{claim*}{Claim}
\newtheorem{cor}[thm]{Corollary} \newtheorem*{cor*}{Corollary}
\newtheorem{prop}[thm]{Proposition} \newtheorem*{prop*}{Proposition}

\theoremstyle{definition}
\newtheorem{defn}[thm]{Definition} \newtheorem*{defn*}{Definition}

\theoremstyle{remark}
\newtheorem{rem}[thm]{Remark} \newtheorem*{rem*}{Remark}
 \newtheorem*{example*}{Example}
\newtheorem{question}[thm]{Question} \newtheorem*{question*}{Question}

\newcommand{\powerset}{\mathcal{P}}

\newcommand{\Ord}{\mathrm{Ord}}

\newcommand{\forces}{\mathbin{\Vdash}}

\newcommand{\bfSigma}{\utilde{\bf{\Sigma}}}
\newcommand{\bfPi}{\utilde{\bf{\Pi}}}

\newcommand{\bfDelta}{\utilde{\bf{\Delta}}}
\newcommand{\jhat}{\mathbf {\hat \jmath}}

\DeclareMathOperator{\p}{p}

\DeclareMathOperator{\Col}{Col}
\DeclareMathOperator{\crit}{crit}

\DeclareMathOperator{\ran}{ran}

\DeclareMathOperator{\otp}{o.\!t.}
\DeclareMathOperator{\uB}{uB}

\begin{document}

\author{Ralf Schindler and Trevor M.\ Wilson}
\title{The consistency strength of the perfect set property for universally Baire sets of reals}

\address[R.\ Schindler]{Institut f\"{u}r mathematische Logik und Grundlagenforschung\\Universit\"{a}t M\"{u}nster\\Einsteinstr.\ 62, 48149\\M\"{u}nster, Germany}
\email{rds@math.uni-muenster.de} 
\urladdr{https://ivv5hpp.uni-muenster.de/u/rds/}

\address[T.\ Wilson]{Department of Mathematics\\Miami University\\Oxford, Ohio 45056\\USA}
\email{twilson@miamioh.edu} 
\urladdr{https://www.users.miamioh.edu/wilso240}

\begin{abstract}
 We show that the statement ``every universally Baire set of reals has the perfect set property'' is equiconsistent modulo ZFC with the existence of a cardinal that we call a virtually Shelah cardinal.  These cardinals resemble Shelah cardinals but are much weaker: if $0^\sharp$ exists then every Silver indiscernible is virtually Shelah in $L$.  We also show that the statement $\uB = {\bfDelta}^1_2$, where $\uB$ is the pointclass of all universally Baire sets of reals, is equiconsistent modulo ZFC with the existence of a $\Sigma_2$-reflecting virtually Shelah cardinal.
\end{abstract}

\maketitle

\section{Introduction}

A set of reals, meaning a subset of the Baire space $\omega^\omega$, is called \emph{universally Baire} if its preimages under all continuous functions from all topological spaces have the Baire property (Feng, Magidor and Woodin \cite{FenMagWoo}.) We denote the pointclass of all universally Baire sets of reals by \emph{uB}. The universally Baire sets of reals include the ${\bfSigma}^1_1$ (analytic) and ${\bfPi}^1_1$ (coanalytic) sets of reals, but not necessarily the ${\Delta^1_2}$ sets of reals. Every universally Baire set of reals is Lebesgue measurable and has the Baire property. If there is a Woodin cardinal then every universally Baire set of reals has the perfect set property; on the other hand, if $\omega_1^L = \omega_1$ then there is a set of reals that is $\Pi^1_1$, hence universally Baire, and fails to have the perfect set property (G\"{o}del; see Kanamori \cite[Theorem 13.12]{KanHigherInfinite}.)

In this article we will characterize the consistency strength of the theory ZFC + ``every universally Baire set of reals has the perfect set property''  using a virtual large cardinal property that we call the virtual Shelah property. We work in ZFC unless otherwise stated.

Many large cardinal properties are defined in terms of elementary embeddings $j : M \to N$ where $M$ and $N$ are structures. (If $M$ and $N$ are sets with no structure given, we consider them as structures with the $\in$ relation.)  Such a definition can be weakened to a ``virtual'' large cardinal property by only requiring $j$ to exist in some generic extension of $V$.

Examples of virtual large cardinal properties are remarkability, which is a virtual form of Magidor's characterization of supercompactness (Schindler \cite[Lemma 1.6]{SchProperForcingRemarkableII},) and the generic Vop\v{e}nka Principle, which is a virtual form of Vop\v{e}nka's Principle (Bagaria, Gitman, and Schindler \cite{BagGitSchGenericVopenka}.) For an introduction to virtual large cardinal properties, see Gitman and Schindler \cite{GitSchVirtualLargeCardinals}.  For an application of virtual large cardinals to descriptive set theory involving $\aleph_1$-Suslin sets instead of universally Baire sets, see Wilson \cite{WilGenericVopenka}.

 For structures $M$ and $N$ in $V$, if some generic extension of $V$ contains an elementary embedding of $M$ into $N$ then by the absoluteness of elementary embeddability of countable structures (see Bagaria, Gitman, and Schindler \cite[Lemma 2.6]{BagGitSchGenericVopenka}) every generic extension of $V$ by the poset $\Col(\omega,\mathcal{M})$ contains an elementary embedding of $M$ into $N$.  The converse implication holds also, of course.  We abbreviate these equivalent conditions by the phrase ``there is a generic elementary embedding of $M$ into $N$.''
 
\begin{defn}\label{defn:kappa-is-crit-j}
 A cardinal $\kappa$ is \emph{virtually Shelah} if for every function $f : \kappa \to \kappa$ there is an ordinal $\lambda > \kappa$, a transitive set $M$ with $V_\lambda \subset M$, and a generic elementary embedding $j : V_\lambda \to M$ with $\crit(j) = \kappa$ and $j(f)(\kappa) \le \lambda$.
\end{defn}

The virtual Shelah property follows immediately from the Shelah property but is much weaker: Proposition \ref{prop:silver-indiscernibles-are-v-shelah} will show that if $0^\sharp$ exists then every Silver indiscernible is virtually Shelah in $L$, a result that is typical of virtual large cardinal properties.

Our main result is stated below. It relates models of ZFC with virtually Shelah cardinals to models of ZFC in which the universally Baire sets are few in number and have nice properties. It also includes a combinatorial statement about the order type of countable sets of ordinals.

\begin{thm}\label{thm:v-Shelah}
 The following statements are equiconsistent modulo ZFC.
 \begin{enumerate}
  \item\label{item:equi-v-Shelah} There is a virtually Shelah cardinal.
  \item\label{item:omega-1-uB-sets} $\left|\uB\right| = \omega_1$.
  \item\label{item:equi-L-R-uB-LM} Every set of reals in $L(\mathbb{R},\uB)$ is Lebesgue measurable.
  \item\label{item:equi-L-R-uB-PSP} Every set of reals in $L(\mathbb{R},\uB)$ has the perfect set property.
  \item\label{item:equi-uB-PSP} Every universally Baire set of reals has the perfect set property.
  \item\label{item:equi-combinatorial} For every function $f: \omega_1 \to \omega_1$ there is an ordinal $\lambda > \omega_1$ such that for a stationary set of $\sigma \in \powerset_{\omega_1}(\lambda)$ we have $\sigma \cap \omega_1 \in \omega_1$ and $\otp(\sigma) \ge f(\sigma \cap \omega_1)$.
 \end{enumerate}
\end{thm}

We will show that if statement \ref{item:equi-v-Shelah} holds then statements \ref{item:omega-1-uB-sets}, \ref{item:equi-L-R-uB-LM}, and \ref{item:equi-L-R-uB-PSP} hold after forcing with the Levy collapse poset to make a virtually Shelah cardinal equal to $\omega_1$.  Clearly statement \ref{item:equi-L-R-uB-PSP} implies statement \ref{item:equi-uB-PSP}, and we will show that statements \ref{item:omega-1-uB-sets} and \ref{item:equi-L-R-uB-LM} also imply statement \ref{item:equi-uB-PSP}. We will show that statement \ref{item:equi-uB-PSP} implies statement \ref{item:equi-combinatorial}. Finally, we will show that if statement \ref{item:equi-combinatorial} holds, then statement \ref{item:equi-v-Shelah} holds in $L$ as witnessed by the cardinal $\omega_1^V$.

We will also prove an equiconsistency result at a slightly higher level of consistency strength, namely that of a $\Sigma_2$-reflecting virtually Shelah cardinal. A cardinal $\kappa$ is called \emph{$\Sigma_n$-reflecting} if it is inaccessible and $V_\kappa \prec_{\Sigma_n} V$. This definition is particularly natural in the case $n = 2$ because the $\Sigma_2$ statements about a parameter $a$ are the statements that can be expressed in the form ``there is an ordinal $\lambda$ such that $V_\lambda \models \varphi[a]$'' where $\varphi$ is a formula in the language of set theory. Because the existence of a virtually Shelah cardinal above any given cardinal $\alpha$ is a $\Sigma_2$ statement about $\alpha$, if $\kappa$ is a $\Sigma_2$-reflecting virtually Shelah cardinal then $V_\kappa$ satisfies ZFC + ``there is a proper class of virtually Shelah cardinals.''

The existence of a $\Sigma_2$-reflecting cardinal is equiconsistent with the statement ${\bfDelta}^1_2 \subset \uB$ by Feng, Magidor, and Woodin \cite[Theorem 3.3]{FenMagWoo}, who showed that if $\kappa$ is $\Sigma_2$-reflecting then ${\bfDelta}^1_2 \subset \uB$ holds after the Levy collapse forcing to make $\kappa$ equal to $\omega_1$, and conversely that if ${\bfDelta}^1_2 \subset \uB$ then $\omega_1^V$ is $\Sigma_2$-reflecting in $L$. The reverse inclusion $\uB \subset {\bfDelta}^1_2$ is consistent relative to ZFC by Larson and Shelah \cite{LarSheUniversallyMeasurable}, who showed that if $V = L[x]$ for some real $x$ then there is a proper forcing extension in which every universally measurable set of reals (and hence every universally Baire set of reals) is ${\bfDelta}^1_2$.

Combining the argument of Feng, Magidor, and Woodin with parts of the proof of Theorem \ref{thm:v-Shelah}, we will show:

\begin{thm}\label{thm:delta-1-2-equals-uB}
The following statements are equiconsistent modulo ZFC.
 \begin{enumerate}
  \item\label{item:Sigma-2-reflecting-v-Shelah-2} There is a $\Sigma_2$-reflecting virtually Shelah cardinal.
  \item\label{item:uB-equals-Delta-1-2} $\uB = {\bfDelta}^1_2$.
 \end{enumerate}
\end{thm}

The remaining sections of this paper are outlined as follows. In Section \ref{sec:consequences-of-v-shelah} we will prove some consequences of Definition \ref{defn:kappa-is-crit-j} that will be needed for our main results, including a reformulation of the definition in which $\kappa$ is $j(\crit(j))$ instead of $\crit(j)$ (Proposition \ref{prop:downward-reformulation}.) In Section \ref{section:relation-to-other-LC} we will prove some relations between virtually Shelah cardinals and other large cardinals that will not be needed for our main results.  In Section \ref{sec:thin-uB-sets} we will review some properties of universally Baire sets and establish some equivalent conditions for a universally Baire set to be thin, meaning to contain no perfect subset. In Section \ref{section:proof-of-thm-v-Shelah} we will prove Theorem \ref{thm:v-Shelah}. In Section \ref{section:proof-of-thm-delta-1-2-equals-uB} we will prove Theorem \ref{thm:delta-1-2-equals-uB}.

\section{Consequences of the virtual Shelah property}\label{sec:consequences-of-v-shelah}

The following result is typical of virtual large cardinals (see Schindler \cite[Lemma 1.4]{SchProperForcingRemarkableII}.)

\begin{prop}\label{prop:ineffable}
 Every virtually Shelah cardinal is ineffable.
\end{prop}
\begin{proof}
 Let $\kappa$ be a virtually Shelah cardinal. Then there is an ordinal $\lambda > \kappa$, a transitive set $M$ such that $V_\lambda \subset M$, and a generic elementary embedding $j : V_\lambda \to M$ with $\crit(j) = \kappa$.  (Here we will not need $j(f)(\kappa) \le \lambda$ for any particular function $f$.)
 
 Let $\vec{A}$ be a $\kappa$-sequence of sets such that $\vec{A}(\alpha) \subset \alpha$ for every ordinal $\alpha < \kappa$. Then we may define a subset $A \subset \kappa$ by $A = j(\vec{A})(\kappa)$. We will show that the set 
 \[S = \{\alpha < \kappa : A \cap \alpha = \vec{A}(\alpha)\}\]
 is stationary. Letting $C$ be a club set in $\kappa$ we have $\kappa \in j(C)$, and because 
 \[j(A) \cap \kappa = A = j(\vec{A})(\kappa)\]
 we have $\kappa \in j(S)$ also, so it follows that
 \[\kappa \in j(C) \cap j(S) = j(C \cap S)\]
 and by the elementarity of $j$ we have $C \cap S \ne \emptyset$.
\end{proof}

A better lower bound for the consistency strength of virtually Shelah cardinals will be given in Section \ref{section:relation-to-other-LC} along with an upper bound. For our main results we will only need the fact that every virtually Shelah cardinal is an inaccessible limit of inaccessible cardinals, which is a consequence of Proposition \ref{prop:ineffable}.

The following lemma shows (among other things) that the domain of a generic elementary embedding witnessing the virtual Shelah property may be taken to be an inaccessible rank initial segment of $V$, which implies that the domain and codomain both satisfy ZFC.
 
\begin{lem}\label{lem:zfc-models}
 Let $\kappa$ be a virtually Shelah cardinal and let $f : \kappa \to \kappa$. Then there is an inaccessible cardinal $\lambda > \kappa$, a transitive model $M$ of ZFC with $V_\lambda \subset M$, and a generic elementary embedding $j : V_\lambda \to M$ with $\crit(j) = \kappa$ and $j(f)(\kappa) < \lambda < j(\kappa)$.
\end{lem}
\begin{proof}
 Because $\kappa$ is a limit of inaccessible cardinals we may define a function $g : \kappa \to \kappa$ such that $g(\alpha)$ is the least inaccessible cardinal greater than $\max(f(\alpha), \alpha)$ for all $\alpha < \kappa$. Because $\kappa$ is virtually Shelah with respect to the function $g+1$ there is an ordinal $\beta > \kappa$, a transitive set $N$ with $V_\beta \subset N$, and a generic elementary embedding
 \[j: V_\beta \to N \text{ with } \crit(j) = \kappa \text{ and } j(g)(\kappa) < \beta.\]
 By the definition of $g$ from $f$ and the elementarity of $j$ it follows that $j(g)(\kappa)$ is the least inaccessible cardinal in $N$ greater than $\max(j(f)(\kappa), \kappa)$. Because $j(g)$ is a function from $j(\kappa)$ to $j(\kappa)$ we have $j(g)(\kappa) < j(\kappa)$. Letting $\lambda = j(g)(\kappa)$ we therefore have $\lambda > \kappa$ and
 \[j(f)(\kappa) < \lambda < j(\kappa).\]
 Because $\lambda < \beta$ and $V_\beta \subset N$, the inaccessibility of $\lambda$ is absolute from $N$ to $V$. Define $j_1 = j \restriction V_\lambda$ and $M$ = $j(V_\lambda) = V_{j(\lambda)}^N$. Then $j_1 : V_\lambda \to M$ is a generic elementary embedding with $\crit(j_1) = \kappa$ and $j_1(f)(\kappa) < \lambda < j_1(\kappa)$ as desired. Because $V_\lambda$ satisfies ZFC it follows by the elementarity of $j_1$ that $M$ satisfies ZFC.
\end{proof}

It follows from Lemma \ref{lem:zfc-models} that every virtually Shelah cardinal has an inaccessible cardinal above it. Because inaccessibility is preserved by small forcing, combining this fact with the proof of Theorem \ref{thm:v-Shelah} (as outlined following the statement of the theorem) yields the following curious consequence.

\begin{prop}\label{prop:curious}
The following statements are equiconsistent modulo ZFC.
\begin{enumerate}
 \item\label{item:psp-uB} Every universally Baire set of reals has the perfect set property.
 \item\label{item:psp-uB-and-inacc} Every universally Baire set of reals has the perfect set property and there is an inaccessible cardinal.
\end{enumerate}
\end{prop}

Note that we may have $\uB \subsetneq \uB^{V_\lambda}$ for an inaccessible cardinal $\lambda$, so 
the statement ``every universally Baire set of reals has the perfect set property'' may fail to be absolute to $V_\lambda$.

Like other virtual large cardinal properties, the virtual Shelah property is downward absolute to $L$:

\begin{prop}\label{lem:absolute-to-L}
 Every virtually Shelah cardinal is virtually Shelah in $L$.
\end{prop}
\begin{proof}
 Let $\kappa$ be a virtually Shelah cardinal and let $f : \kappa \to \kappa$ be a function in $L$. Then by Lemma \ref{lem:zfc-models} there is an inaccessible cardinal $\lambda > \kappa$, a transitive model $M$ of ZFC with $V_\lambda \subset M$, and a generic elementary embedding $j : V_\lambda \to M$ with $\crit(j) = \kappa$ and $j(f)(\kappa) < \lambda < j(\kappa)$. Note that $L^{V_\lambda} = L_\lambda$ and $L^M = L_\theta$ where $\theta = \Ord^M$. Moreover, because $\lambda$ is inaccessible we have $L_\lambda = V_\lambda^L$. For the elementary embedding $j_1  = j \restriction V_\lambda^L$ we have
 \[j_1 : V_\lambda^L \to L_\theta \text{ and } \crit(j_1) = \kappa \text{ and } j_1(f) = j(f).\]
 Let $G \subset \Col(\omega,V_\lambda^L)$ be a $V$-generic filter. Then by the absoluteness of elementary embeddability of countable structures there is an elementary embedding $j_2 \in L[G]$ such that
  \[j_2 : V_\lambda^L \to L_\theta \text{ and } \crit(j_2) = \kappa \text{ and } j_2(f) = j(f).\]
 We have $V_\lambda^L = L_\lambda \subset L_\theta$ and  $j_2(f)(\kappa) = j(f)(\kappa) < \lambda$, so $j_2$ witnesses the virtual Shelah property for $\kappa$ in $L$ with respect to $f$.
\end{proof} 

The virtual Shelah property has a reformulation in which $\kappa$ is the image of the critical point, as in the characterizations of supercompactness and remarkability by Magidor \cite[Theorem 1]{MagCombinatorialCharacterization} and Schindler \cite[Lemma 1.6]{SchProperForcingRemarkableII} respectively:

\begin{prop}\label{prop:downward-reformulation}
 For every cardinal $\kappa$ the following statements are equivalent:
 \begin{enumerate}
  \item\label{item:kappa-is-crit-j} $\kappa$ is virtually Shelah.
  \item\label{item:kappa-is-image-of-crit-j} For every function $f : \kappa \to \kappa$ there are ordinals $\bar{\lambda}$ and $\lambda$ and a generic elementary embedding $j :V_{\bar{\lambda}} \to V_\lambda$ with $j(\crit(j)) = \kappa$ and $f \in \ran(j)$ and $f(\crit(j)) \le \bar{\lambda}$.\footnote{This statement implies $\lambda > \kappa$, and by restricting $j$ if necessary we may assume $\bar{\lambda} < \kappa$ if desired.}
 \end{enumerate}
\end{prop}
\begin{proof}
 \eqref{item:kappa-is-crit-j} $\implies$ \eqref{item:kappa-is-image-of-crit-j}:
 Assume that $\kappa$ is virtually Shelah and let $f: \kappa \to \kappa$. By Lemma \ref{lem:zfc-models} there is an inaccessible cardinal $\lambda > \kappa$, a transitive model $M$ of ZFC with $V_\lambda \subset M$, and a generic elementary embedding
 \[j : V_\lambda \to M \text{ with }\crit(j) = \kappa \text{ and }j(f)(\kappa) < \lambda < j(\kappa).\]
 Define $\beta = \max(j(f)(\kappa), \kappa+1)$, so $\kappa < \beta < \lambda$ and $V_\beta^M = V_\beta$.
 Let $j_1 = j \restriction V_\beta$ and note that
 \[j_1 : V_\beta \to V_{j(\beta)}^M \text{ and } \crit(j_1) = \kappa \text{ and } j_1(\kappa) = j(\kappa) \text{ and } j_1(f) = j(f).\]
 Let $G \subset \Col(\omega,V_\beta)$ be a $V$-generic filter.
 Then by the absoluteness of elementary embeddability of countable structures there is an elementary embedding $j_2 \in M[G]$ such that 
 \[j_2 : V_\beta \to V_{j(\beta)}^M \text{ and } \crit(j_2) = \kappa \text{ and } j_2(\kappa) = j(\kappa) \text{ and } j_2(f) = j(f).\]
 Then we have
 \[j_2(\crit(j_2)) = j(\kappa) \text{ and } j(f) \in \ran(j_2) \text{ and }j(f)(\crit(j_2)) \le \beta\]
 because $j(f)(\kappa) \le \beta$. By the elementarity of $j$ and the definability of forcing there is an ordinal $\bar{\beta} < \lambda$ such that, letting $g \subset \Col(\omega,V_{\bar{\beta}})$ be a $V$-generic filter, there is an elementary embedding $j_3 \in V_\lambda[g]$ with
 \[j_3 : V_{\bar{\beta}} \to V_{\beta} \text{ and } j_3(\crit(j_3)) = \kappa \text{ and } f \in \ran(j_3) \text{ and } f(\crit(j_3)) \le \bar{\beta}.\]
 Therefore statement \ref{item:kappa-is-image-of-crit-j} holds for $f$.

 \eqref{item:kappa-is-image-of-crit-j} $\implies$ \eqref{item:kappa-is-crit-j}:
 Assume that statement \ref{item:kappa-is-image-of-crit-j} holds and let $f: \kappa \to \kappa$. Note that statement \ref{item:kappa-is-image-of-crit-j} implies the ineffability of $\kappa$ by an argument similar to Proposition \ref{prop:ineffable}. It follows that $\kappa$ is a limit of inaccessible cardinals, so by increasing the values of $f$ we may assume that for all $\alpha < \kappa$, $f(\alpha)$ is an inaccessible cardinal greater than $\alpha$ and we may furthermore define $f^+(\alpha)$ to be the least inaccessible cardinal greater than $f(\alpha)$. Applying statement \ref{item:kappa-is-image-of-crit-j} to the function $f^+: \kappa \to \kappa$ yields ordinals $\bar{\lambda}$ and $\lambda$ and a generic elementary embedding
 \[j :V_{\bar{\lambda}} \to V_\lambda \text{ with } j(\bar{\kappa}) = \kappa \text{ and } f^+ \in \ran(j) \text{ and } f^+(\bar{\kappa}) \le \bar{\lambda}\]
 where $\bar{\kappa} = \crit(j)$. By restricting $j$ if necessary we may assume that $\bar{\lambda}$ is equal to $f^+(\bar{\kappa})$ and is therefore an inaccessible cardinal less than $\kappa$. Let $\bar{\beta} = f(\bar{\kappa})$ and $\beta = j(\bar{\beta})$.
 
 Note that $f^+ \in \ran(j)$ implies $f \in \ran(j)$ because for all $\alpha < \kappa$, $f(\alpha)$ is definable in $V_\lambda$ as the largest inaccessible cardinal less than $f^+(\alpha)$, so we may define $\bar{f} = j^{-1}(f)$.
 Then we have $\bar{f} : \bar{\kappa} \to \bar{\kappa}$ and $\bar{\beta} = j(\bar{f})(\bar{\kappa})$. Note that
 \[ \bar{\kappa} < \bar{\beta} < \bar{\lambda} < \kappa < \beta < \lambda.\]
 For the elementary embedding $j_1 = j \restriction V_{\bar{\beta}}$ we have
 \[ j_1 : V_{\bar{\beta}} \to V_\beta \text{ and } \crit(j_1) = \bar{\kappa} \text { and } j_1(\bar{\kappa}) = \kappa \text{ and } j_1(\bar{f}) = f.\]
 Let $g \subset \Col(\omega, V_{\bar{\beta}})$ be a $V$-generic filter. Then by the absoluteness of elementary embeddability of countable structures there is an elementary embedding $j_2 \in V[g]$ such that
 \[ j_2 : V_{\bar{\beta}} \to V_\beta \text{ and } \crit(j_2) = \bar{\kappa} \text { and } j_2(\bar{\kappa}) = \kappa \text{ and } j_2(\bar{f}) = f.\]
 Letting $M = V_\beta$ we therefore have 
 \[ V_{\bar{\beta}} \subset M \text{ and } j_2 : V_{\bar{\beta}} \to M \text{ and } \crit(j_2) = \bar{\kappa} \text { and } j_2(\bar{f})(\bar{\kappa}) = \bar{\beta}.\]
 Because $\bar{\lambda}$ is inaccessible, a Skolem hull argument in $V[g]$ yields a transitive set $M' \in V_{\bar{\lambda}}[g]$ and an elementary embedding $j_3 \in V_{\bar{\lambda}}[g]$ such that
 \[ V_{\bar{\beta}} \subset M' \text{ and } j_3 : V_{\bar{\beta}} \to M' \text{ and } \crit(j_3) = \bar{\kappa} \text { and } j_3(\bar{f})(\bar{\kappa}) = \bar{\beta}.\]
 Let $G \subset \Col(\omega, V_\beta)$ be a $V$-generic filter. By the elementarity of $j$ and the definability of forcing, there is a transitive set $M'' \in V_{\lambda}[G]$ and an elementary embedding $j_4 \in V_{\lambda}[G]$ with
 \[V_{\beta} \subset M'' \text{ and } j_4 : V_{\beta} \to M'' \text{ and } \crit(j_4) = \kappa \text{ and } j_4(f)(\kappa) = \beta.\]
 Therefore $\kappa$ is virtually Shelah with respect to $f$.
\end{proof}

\begin{rem}
 If statement \ref{item:kappa-is-image-of-crit-j} of Proposition \ref{prop:downward-reformulation} is ``unvirtualized'' by requiring  $j$ to exist in $V$, the resulting statement is equivalent to the Shelah-for-supercompactness property defined by Perlmutter \cite[Definition 2.7]{PerSupercompactAlmostHuge}. Virtually Shelah cardinals are therefore also virtually Shelah-for-supercompactness. This is similar to the fact that remarkable cardinals can be considered either as virtually strong cardinals or as virtually supercompact cardinals.
\end{rem}

\section{Relation to other large cardinal properties}\label{section:relation-to-other-LC}

It is clear from the definition that every Shelah cardinal is virtually Shelah. In fact the virtual Shelah property is much weaker than the Shelah property by the following result, which is typical of virtual large cardinal properties:

\begin{prop}\label{prop:silver-indiscernibles-are-v-shelah}
 If $0^\sharp$ exists then every Silver indiscernible is virtually Shelah in $L$.
\end{prop}
\begin{proof}
 Assume that $0^\sharp$ exists and let $\kappa$ be a Silver indiscernible.  Then there is an elementary embedding $j: L \to L$ with $\crit(j) = \kappa$.  Let $f : \kappa \to \kappa$ be a function in $L$ and define $\lambda = \max(j(f)(\kappa), \kappa+1)$.
 For the elementary embedding $j_1 = j \restriction V_\lambda^L$ we have
 \[ j_1 : V_\lambda^L \to V_{j(\lambda)}^L \text{ and } \crit(j_1) = \kappa \text{ and } j_1(\kappa) = j(\kappa) \text{ and } j_1(f) = j(f).\]
 Let $G \subset \Col(\omega,V_\lambda^L)$ be a $V$-generic filter.
 Then by the absoluteness of elementary embeddability of countable structures there is an elementary embedding $j_2 \in L[G]$ such that
 \[ j_2 : V_\lambda^L \to V_{j(\lambda)}^L \text{ and } \crit(j_2) = \kappa \text{ and } j_2(\kappa) = j(\kappa) \text{ and } j_2(f) = j(f).\]
 We have $j_2(f)(\kappa) = j(f)(\kappa) \le \lambda$, so $j_2$ witnesses the virtual Shelah property for $\kappa$ in $L$ with respect to the function $f$.
\end{proof}

We can obtain a better upper bound for the consistency strength of virtually Shelah cardinals in terms of the hierarchy of $\alpha$-iterable cardinals defined by Gitman \cite{GitRamsey}. If $0^\sharp$ exists then every Silver indiscernible is $\alpha$-iterable in $L$ for every ordinal $\alpha < \omega_1^L$ by Gitman and Welch \cite[Theorem 3.11]{GitWelRamseyII}, and we will show that 2-iterable cardinals already exceed virtually Shelah cardinals in consistency strength. We will not need the definition of $\alpha$-iterability below, only a certain established consequence of it in the case $\alpha = 2$.
 
\begin{prop}\label{prop:2-iterable-stationary-limit}
 If $\kappa$ is a $2$-iterable cardinal then $\kappa$ is a stationary limit of cardinals that are virtually Shelah in $V_\kappa$.
\end{prop}
\begin{proof}
 Assume that $\kappa$ is 2-iterable. Then by Gitman and Welch \cite[Theorem 4.7]{GitWelRamseyII} there is a transitive model $M$ of ZFC with $V_\kappa \in M$ and an elementary embedding $j : M \to N$ with critical point $\kappa$ where $N$ is a transitive model of ZFC and $M = V_{j(\kappa)}^N$.
  
 First, we show that $\kappa$ is virtually Shelah in $N$.  Let $f : \kappa \to \kappa$ be in $N$ and therefore also in $M$.  Define $\lambda = \max(j(f)(\kappa), \kappa + 1)$.  Because $j(f)$ is a function from $j(\kappa)$ to $j(\kappa)$ we have $\lambda < j(\kappa) = \Ord^M$, so $V_\lambda^M = V_\lambda^N$ and we have an elementary embedding $j_1 = j \restriction V_\lambda^M$ with
 \[ j_1 : V_\lambda^N \to V_{j(\lambda)}^N \text{ and } \crit(j_1) = \kappa \text{ and } j_1(\kappa) = j(\kappa)\text{ and } j_1(f) = j(f).\]
 Let $G \subset \Col(\omega,V_\lambda^N)$ be a $V$-generic filter.
 Then by the absoluteness of elementary embeddability of countable structures there is an elementary embedding $j_2 \in N[G]$ such that
 \[ j_2 : V_\lambda^N \to V_{j(\lambda)}^N \text{ and } \crit(j_2) = \kappa \text{ and } j_2(\kappa) = j(\kappa)\text{ and } j_2(f) = j(f).\]
 Because $j_2(f)(\kappa) = j(f)(\kappa) \le \lambda$, this elementary embedding $j_2$ witnesses the virtual Shelah property for $\kappa$ in $N$ with respect to $f$.

 Now let $C$ be club in $\kappa$. Then $\kappa \in j(C)$, so the model $N$ satisfies the statement ``there is a virtually Shelah cardinal in $j(C)$'' and by the elementarity of $j$ it follows that the model $M$ satisfies the statement  ``there is a virtually Shelah cardinal in $C$.'' Let $\bar{\kappa} \in C$ be virtually Shelah in $M$. Because we have $j(\bar{\kappa}) = \bar{\kappa}$ and 
 \[j(V_\kappa) = j(V_\kappa^M) = V_{j(\kappa)}^N = M,\]
 it follows by the elementarity of $j$ that $\bar{\kappa}$ is virtually Shelah in $V_\kappa$.
\end{proof}

Because the set of all ordinals $\alpha< \kappa$ such that $V_\alpha \prec_{\Sigma_2} V_\kappa$ is club in $\kappa$, the conclusion of Proposition \ref{prop:2-iterable-stationary-limit} implies that $\kappa$ is a stationary limit of cardinals that are both $\Sigma_2$-reflecting and virtually Shelah in $V_\kappa$. (Note that the virtual Shelah property is upward absolute from $V_\kappa$ to $V$ but the $\Sigma_2$-reflecting property might not be.) It follows that ZFC + ``there is a $\Sigma_2$-reflecting virtually Shelah cardinal'' has lower consistency strength than ZFC + ``there is a 2-iterable cardinal.''

Recall that every virtually Shelah cardinal is ineffable by Proposition \ref{prop:ineffable}. We can obtain a better lower bound for the consistency strength of virtually Shelah cardinals in terms of the virtually $A$-extendible cardinals defined by Gitman and Hamkins \cite[Definition 6]{GitHamOrdNotDelta2Mahlo}: For a cardinal $\alpha$ and a class $A$ (meaning either a definable class in ZFC or an arbitrary class in GBC) we say that $\alpha$ is \emph{virtually $A$-extendible} if for every ordinal $\beta > \alpha$ there is an ordinal $\theta$ and a generic elementary embedding 
\[j : (V_\beta; \mathord{\in},A \cap V_\beta) \to (V_\theta; \mathord{\in},A \cap V_\theta)\]
with $ \crit(j) = \alpha$ and $j(\alpha) > \beta$.
 
\begin{prop}\label{prop:virtually-A-extendible}
 If $\kappa$ is a virtually Shelah cardinal then the structure $(V_\kappa, V_{\kappa+1}; \mathord{\in})$ satisfies the statement ``for every class $A$ there is a virtually $A$-extendible cardinal.''
\end{prop}
\begin{proof}
 Let $\kappa$ be a virtually Shelah cardinal, let $A \subset V_\kappa$, and assume toward a contradiction that no cardinal less than $\kappa$ is virtually $A$-extendible in $(V_\kappa, V_{\kappa+1}; \mathord{\in})$. Then we may define a function $f:\kappa \to \kappa$ such that for every ordinal $\alpha < \kappa$, $f(\alpha)$ is the least ordinal greater than $\alpha$ such that for every ordinal $\theta < \kappa$ there is no generic elementary embedding from $(V_{f(\alpha)}; \mathord{\in},A \cap V_{f(\alpha)})$ to $(V_\theta; \mathord{\in},A \cap V_\theta)$ that has critical point $\alpha$ and maps $\alpha$ above $f(\alpha)$. 
 
 Because $\kappa$ is virtually Shelah, by Lemma \ref{lem:zfc-models} there is an inaccessible cardinal $\lambda > \kappa$, a transitive model $M$ of ZFC with $V_\lambda \subset M$, and a generic elementary embedding 
 \[j : V_\lambda \to M \text{ with } \crit(j) = \kappa \text { and }j(f)(\kappa) < \lambda < j(\kappa).\]
 Defining $\beta = j(f)(\kappa)$, we have
 \[\kappa < \beta < \lambda < j(\kappa) < j(\beta) < \Ord^M.\]
 Let $j_1 = j \restriction V_\beta$, considered as an elementary embedding whose domain is the structure $(V_\beta; \mathord{\in}, j(A) \cap V_\beta)$.  Then we have
 \[j_1 : (V_\beta; \mathord{\in}, j(A) \cap V_\beta) \to (V_{j(\beta)}^M; \mathord{\in}, j(j(A) \cap V_\beta)) \text{ and } \crit(j_1) = \kappa \text{ and } j_1(\kappa) = j(\kappa).\]
 Let $G \subset \Col(\omega,V_\beta)$ be a $V$-generic filter.
 Then by the absoluteness of elementary embeddability of countable structures there is an elementary embedding $j_2 \in M[G]$ such that
 \[j_2 : (V_\beta; \mathord{\in}, j(A) \cap V_\beta) \to (V_{j(\beta)}^M; \mathord{\in}, j(j(A) \cap V_\beta)) \text{ and } \crit(j_2) = \kappa \text{ and } j_2(\kappa) = j(\kappa).\]
 Note that $j_2(\kappa) = j(\kappa) > j(f)(\kappa) = \beta$.
 By the elementarity of $j$ and the definability of forcing it follows that there is a cardinal $\alpha < \kappa$ such that, letting $g \subset \Col(\omega,V_{f(\alpha)})$ be a $V$-generic filter, there is an elementary embedding $j_3 \in V_\lambda[g]$ with
 \[j_3 : (V_{f(\alpha)}; \mathord{\in}, A \cap V_{f(\alpha)}) \to (V_{\beta}; \mathord{\in}, j(A) \cap V_\beta) \text{ and } \crit(j_3) = \alpha \text{ and } j_3(\alpha)  > f(\alpha).\]
 Because $V_\lambda \subset M$ we have $j_3 \in M[g]$,
 and because $\beta < j(\kappa)$ it then follows by the elementarity of $j$
 that there is an ordinal $\theta < \kappa$ and an elementary embedding $j_4 \in V_\lambda[g]$ such that
 \[j_4 : (V_{f(\alpha)}; \mathord{\in}, A \cap V_{f(\alpha)}) \to (V_{\theta}; \mathord{\in}, A \cap V_\theta) \text{ and } \crit(j_4) = \alpha \text{ and } j_4(\alpha) > f(\alpha).\]
 The existence of such a generic elementary embedding $j_4$ contradicts the definition of $f$.
\end{proof}

We can state a further consequence of the virtual Shelah property in terms of the generic Vop\v{e}nka principle defined by Bagaria, Gitman, and Schindler \cite{BagGitSchGenericVopenka}, which says that for every proper class of structures of the same type there is a generic elementary embedding from one of the structures into another. The generic Vop\v{e}nka principle, formalized as a statement in GBC, follows from the existence of a virtually $A$-extendible cardinal for every class $A$ by Gitman and Hamkins \cite[Theorem 7]{GitHamOrdNotDelta2Mahlo}. Combining this fact with Proposition \ref{prop:virtually-A-extendible} gives the following result.

\begin{cor}
 If $\kappa$ is a virtually Shelah cardinal then the generic Vop\v{e}nka principle holds in the structure $(V_\kappa, V_{\kappa+1}; \mathord{\in})$.
\end{cor}

\begin{rem}
 Gitman and Hamkins \cite[Theorem 7]{GitHamOrdNotDelta2Mahlo} showed more specifically that the generic Vop\v{e}nka principle is equivalent to existence of a weakly virtually $A$-extendible cardinal for every class $A$, where the definition of weak virtual $A$-extendibility is obtained from the definition of virtual $A$-extendibility by removing the condition on the image of the critical point. Solovay, Reinhardt, and Kanamori \cite[Theorem 6.9]{SolReiKanStrongAxioms} proved an analogous result in the non-virtual setting, where Kunen's inconsistency erases the distinction between weak extendibility and extendibility: Vop\v{e}nka's principle is equivalent to the existence of an $A$-extendible cardinal for every class $A$.
\end{rem}

\section{Thin universally Baire sets of reals}\label{sec:thin-uB-sets}
 
In this section we will establish some equivalent conditions for a universally Baire set of reals to be \emph{thin}, meaning to have no perfect subset. To do this we will need a characterization of universal Baireness in terms of trees and forcing due to Feng, Magidor and Woodin \cite{FenMagWoo}.

For a class $X$, a \emph{tree on $X$} is a subset of $X^{\mathord{<}\omega}$ that is closed under initial segments. For a tree $T$ on $X$ we let $[T]$ denote the set of all infinite branches of $T$. Note that $[T]$ is a closed subset of $X^\omega$ (where $X$ has the discrete topology) and conversely every closed subset of $X^\omega$ is the set of branches of some tree on $X$.

We will typically consider trees on $\omega \times \Ord$, whose elements we may think of either as finite sequences of pairs or as pairs of finite sequences.  Because we require trees to be sets, a tree on $\omega \times \Ord$ is actually a tree on $\omega \times \gamma$ for some ordinal $\gamma$, but there is usually no need to specify a particular ordinal.

The letter $\p$ denotes projection:
\[\p[T]= \{x \in \omega^\omega : (x,f) \in [T] \text{ for some }f \in \Ord^\omega\}.\]
For a tree $T$ on $\omega \times \Ord$ and a real $x \in \omega^\omega$ we define the \emph{section} $T_x$ of $T$ as the set of all $s \in \Ord^{\mathord{<}\omega}$ such that $(x \restriction \left| s \right|, s) \in T$. Note that $T_x$ is a tree on $\Ord$ that is illfounded if and only if $x \in \p[T]$, so the statement $x \in \p[T]$ is absolute to all transitive models of ZFC containing $x$ and $T$ by the absoluteness of wellfoundedness.

A pair of trees $(T,\tilde{T})$ on $\omega \times \Ord$ is \emph{complementing} if 
\[\p[T] = \omega^\omega \setminus \p[\tilde{T}],\]
and for a poset $\mathbb{P}$ it is \emph{$\mathbb{P}$-absolutely complementing} if it is complementing in every generic extension of $V$ by $\mathbb{P}$. The statement $\p[T] \cap \p[\tilde{T}] = \emptyset$ is generically absolute by the absoluteness of wellfoundedness of the tree of all triples $(r,s_1,s_2)$ such that $(r, s_1) \in T$ and $(r,s_2) \in \tilde{T}$, so a complementing pair of trees $(T,\tilde{T})$ is $\mathbb{P}$-absolutely complementing if and only if $\p[T] \cup \p[\tilde{T}] = \omega^\omega$ in every generic extension of $V$ by $\mathbb{P}$. A tree $T$ is called \emph{$\mathbb{P}$-absolutely complemented} if there is a tree $\tilde{T}$ such that the pair $(T,\tilde{T})$ is $\mathbb{P}$-absolutely complementing.

We say that a set of reals $A$ is \emph{$\mathbb{P}$-Baire} if $A = \p[T]$ for some $\mathbb{P}$-absolutely complemented tree $T$ on $\omega \times \Ord$. By Feng, Magidor, and Woodin \cite[Theorem 2.1]{FenMagWoo} a set of reals is universally Baire if and only if it is $\mathbb{P}$-Baire for every poset $\mathbb{P}$. We will adopt this characterization of universal Baireness as our definition from now on.

For a cardinal $\kappa$, we say that a set of reals is \emph{$\kappa$-universally Baire} if it is $\mathbb{P}$-Baire for every poset $\mathbb{P}$ of cardinality less than $\kappa$.  We denote the pointclass of all $\kappa$-universally Baire sets of reals by $\uB_\kappa$. It is not hard to see that if $\kappa$ is inaccessible, then a set of reals is $\kappa$-universally Baire if and only if it is $\Col(\omega,\mathord{<}\kappa)$-Baire.

If a set of reals $A$ is $\mathbb{P}$-Baire and $G \subset \mathbb{P}$ is a $V$-generic filter, then the \emph{canonical extension} of $A$ to $V[G]$ is the set of reals $A^{V[G]}$ in $V[G]$ defined by 
\[A^{V[G]} = \p[T]^{V[G]}\]
where $T$ is a $\mathbb{P}$-absolutely complemented tree in $V$ such that $A = \p[T]^V$.  By a standard argument using the absoluteness of wellfoundedness, this definition of the canonical extension does not depend on the choice of $T$.

For every positive integer $n$, all of the above definitions and facts about universally Baire sets of reals can easily be generalized to universally Baire $n$-ary relations on the reals by replacing trees on $\omega \times \Ord$ with trees on $\omega^n \times \Ord$.

Now we can establish some equivalent conditions for thinness. The equivalence of statements \ref{item:thin} and \ref{item:no-new-elements} below seems to be well-known (and perhaps the others are also) but we are not aware of a reference.
  
\begin{lem}\label{lem:thin-equivalences}
For every universally Baire set of reals $A$, the following statements are equivalent in ZFC.
\begin{enumerate}
 \item\label{item:thin} $A$ is thin.
 \item\label{item:no-new-elements} $A^{V[G]} = A^V$ for every generic extension $V[G]$ of $V$.
 \item\label{item:every-relation-uB} For every $n < \omega$, every subset of $A^n$ is universally Baire.
 \item\label{item:uB-wellordering} There is a universally Baire wellordering of $A$.
\end{enumerate}
\end{lem}
\begin{proof}
 \eqref{item:thin} $\implies$ \eqref{item:no-new-elements}: Assume that statement  \ref{item:no-new-elements} fails, so there is a generic extension $V[G]$ of $V$ by a poset $\mathbb{P}$ and a real $x \in A^{V[G]} \setminus V$.  Letting $(T,\tilde{T})$ be a $\mathbb{P}$-absolutely complementing pair of trees for $A$ in $V$ we have $x \in \p[T]^{V[G]} \setminus L[T]$, so by Mansfield's theorem in $V[G]$ (see Jech \cite[Lemma 25.24]{Jec2002}) there is a perfect tree $U \in L[T]$ on $\omega$ such that $[U] \subset \p[T]$ in $V[G]$. This implies $[U] \subset A$ in $V$, so $A$ is not thin.
 
 \eqref{item:no-new-elements} $\implies$ \eqref{item:every-relation-uB}:
 Assume statement \ref{item:no-new-elements}. Then there is a tree $T_{\neg A}$ on $\omega \times \Ord$ such that
 \[{} \forces_{\mathbb{P}} \p[T_{\neg A}] = \omega^\omega \setminus A.\]
 From this tree, one can define a tree $T_{\neg A^n}$ on $\omega^n \times \Ord$ such that
 \[{} \forces_{\mathbb{P}} \p[T_{\neg A^n}] = (\omega^\omega)^n \setminus A^n.\]
 Now let $n < \omega$, let $B \subset A^n$, and let $\mathbb{P}$ be a poset. We will show that $B$ is $\mathbb{P}$-Baire. Take trees $T_B$ and $T_{A^n \setminus B}$ on $\omega^n \times \left|B\right|$ and $\omega^n \times \left|A^n \setminus B\right|$ respectively
 that project to  $B$ and $A^n \setminus B$ respectively in every generic extension.  (Such trees can be trivially defined for every pointset.) From $T_{\neg A^n}$ and $T_{A^n \setminus B}$, one can define a tree $T_{\neg B}$ on $\omega^n \times \Ord$ 
 such that every generic extension satisfies $\p[T_{\neg B}] = \p[T_{\neg A^n}] \cup \p[T_{A^n \setminus B}]$. Then we have
 \[{} \forces_{\mathbb{P}} \p[T_{\neg B}] = (\omega^\omega)^n \setminus B,\]
 so the pair of trees $(T_B, T_{\neg{B}})$ witnesses that $B$ is $\mathbb{P}$-Baire.

 \eqref{item:every-relation-uB} $\implies$ \eqref{item:uB-wellordering}: This follows directly from the existence of a wellordering of $A$.

 \eqref{item:uB-wellordering} $\implies$ \eqref{item:thin}: Suppose toward a contradiction that some universally Baire set of reals $A$ has a universally Baire wellordering but is not thin.  Because $A$ has a perfect subset there is a continuous injection $f : 2^\omega \to \omega^\omega$ whose range is contained in $A$. Taking the preimage of a universally Baire wellordering of $A$ under the continuous function $f \times f$ we obtain a wellordering of $2^\omega$ with the Baire property, which leads to a contradiction using the Kuratowski--Ulam theorem (see Kanamori \cite[Corollary 13.10]{KanHigherInfinite}.)
\end{proof}

\section{Proof of Theorem \ref{thm:v-Shelah}}\label{section:proof-of-thm-v-Shelah}
 
The following ``universally Baire reflection'' lemma is the key to obtaining consequences of the virtual Shelah property by forcing. Our statement of the lemma will use the following definition.  For a $V$-generic filter $G$ on the poset $\Col(\omega,\mathord{<}\kappa)$ and an ordinal $\alpha < \kappa$ we define $G \restriction \alpha = G \cap \Col(\omega,\mathord{<}\alpha)$, which is a $V$-generic filter on the poset $\Col(\omega,\mathord{<}\alpha)$.

\begin{lem}\label{lem:forcing}
 Let $\kappa$ be a virtually Shelah cardinal, let $G \subset \Col(\omega,\mathord{<}\kappa)$ be a $V$-generic filter, and let $A$ be a universally Baire set of reals in $V[G]$. Then there is an ordinal $\alpha < \kappa$ and a $\kappa$-universally Baire set of reals $A_0$ in $V[G\restriction\alpha]$ such that $A = A_0^{V[G]}$.
\end{lem}
\begin{proof}
 Suppose toward a contradiction that for every ordinal $\alpha < \kappa$ we have
 \[\forall A_0 \in \uB_\kappa^{V[G\restriction\alpha]}\, A_0^{V[G]} \ne A.\]
 Then for every ordinal $\alpha < \kappa$, because $\kappa$ is inaccessible in $V[G \restriction \alpha]$ and we have
 \[ \uB_\kappa^{V[G \restriction \alpha]} = \bigcap_{\beta < \kappa} \uB_\beta^{V[G \restriction \alpha]} \text{ and } \mathbb{R}^{V[G]} = \bigcup_{\beta < \kappa} \mathbb{R}^{V[G\restriction \beta]},\]
 it follows that for all sufficiently large ordinals $\beta < \kappa$ we have $\beta > \alpha$ and
 \[\forall A_0 \in \uB^{V[G\restriction\alpha]}_{\beta}\, A_0^{V[G \restriction \beta]} \ne A \cap V[G \restriction \beta].\]
 Then by the $\kappa$-chain condition for $\Col(\omega,\mathord{<}\kappa)$  it follows that there is a function $f : \kappa \to \kappa$ in $V$
 such that for all $\alpha < \kappa$ we have $f(\alpha) > \alpha$ and
 \[\forall A_0 \in \uB^{V[G\restriction\alpha]}_{f(\alpha)}\, A_0^{V[G \restriction f(\alpha)]} \ne A \cap V[G \restriction f(\alpha)].\]
 Because $\kappa$ is a limit of inaccessible cardinals in $V$ we may additionally assume that $f(\alpha)$ is inaccessible in $V$ for every ordinal $\alpha < \kappa$.
   
 Because $\kappa$ is virtually Shelah, by Lemma \ref{lem:zfc-models} there is an inaccessible cardinal $\lambda > \kappa$, a transitive model $M$ of ZFC in $V$ with $V_\lambda \subset M$, and a generic elementary embedding 
 \[ j : V_{\lambda} \to M \text{ with } \crit(j) = \kappa \text{ and } j(f)(\kappa) < \lambda < j(\kappa).\]
 Defining $\beta = j(f)(\kappa)$ we have 
 \[\kappa < \beta < \lambda < j(\kappa) < j(\beta) < \Ord^M.\]
 Note that $\beta$ is inaccessible in $M$ by our assumption on $f$ and the elementarity of $j$, and because $V_\lambda \subset M$ and $\beta < \lambda$ this implies that $\beta$ is also inaccessible in $V$.

 We can extend $j$ to a generic elementary embedding
 \[\jhat : V_{\lambda}[G] \to M[H]\]
 where $H\subset \Col(\omega,\mathord{<}j(\kappa))$ is an $V$-generic filter such that $H \restriction \kappa = G$. By the fact that $\beta = j(f)(\kappa)$ and the elementarity of $\jhat$ it follows that
 \[\forall A_0 \in \uB^{M[G]}_{\beta}\,A_0^{M[H \restriction \beta]} \ne \jhat(A) \cap M[H \restriction \beta].\]
 
 We will now obtain a contradiction by showing
 \[A \in \uB^{M[G]}_{\beta} \text{ and } A^{M[H \restriction \beta]} = \jhat(A) \cap M[H \restriction \beta].\]
 In $V[G]$, because $A$ is universally Baire we may take a $\Col(\omega,\mathord{<}\beta)$-absolutely complementing pair of trees $(T,\tilde{T})$ on $\omega \times \Ord$ such that $\p[T] = A$. We may assume that $(T,\tilde{T}) \in V_{\lambda}[G]$: if necessary, replace $T$ and $\tilde{T}$ by their images under the transitive collapse of an elementary substructure of a sufficiently large rank initial segment of $V[G]$ containing $V_\beta[G] \cup \{\beta,T,\tilde{T}\}$ and having cardinality $\beta$. 
 Because $V_\lambda \subset M$ we have $V_\lambda[G] \subset M[G]$ and the pair $(T,\tilde{T}) \in V_{\lambda}[G]$ witnesses $A \in \uB^{M[G]}_{\beta}$ as desired.
 
 We have $(\jhat(T),\jhat(\tilde{T})) \in M[H]$ and it follows by the elementarity of $\jhat$ that \[\p[\jhat(T)]^{M[H]} = \jhat(A).\]
 Mapping branches pointwise by $j$ gives the inclusions
 \[\p[T] \subset \p[\jhat(T)] \text{ and } \p[\tilde{T}] \subset \p[\jhat(\tilde{T})]\]
 in all models containing all of these trees, and particularly in $M[H]$. Because the pair $(T,\tilde{T})$ is $\Col(\omega,\mathord{<}\beta)$-absolutely complementing in $V[G]$ and the pair $(\jhat(T),\jhat(\tilde{T}))$ is complementing in $M[H]$, these inclusions become equalities when restricted to the reals of $M[H \restriction \beta]$, so
 \[A^{M[H \restriction \beta]} = \p[T]^{M[H \restriction \beta]} = \p[\jhat(T)]^{M[H \restriction \beta]} = \jhat(A) \cap M[H \restriction \beta],\]
 completing the desired contradiction.
\end{proof}
 
We can use Lemma \ref{lem:forcing}  to show that the consistency of statement \ref{item:equi-v-Shelah} of Theorem \ref{thm:v-Shelah} (which says there is a virtually Shelah cardinal) implies the consistency of statements \ref{item:omega-1-uB-sets}, \ref{item:equi-L-R-uB-LM}, \ref{item:equi-L-R-uB-PSP}, and the corresponding statement for the Baire property:
 
\begin{prop}\label{prop:forward-direction}
 Let $\kappa$ be a virtually Shelah cardinal and let $G \subset \Col(\omega,\mathord{<}\kappa)$ be a $V$-generic filter.
 Then the following statements hold in $V[G]$:
 \begin{itemize}
  \item $\left|\uB\right| = \omega_1$.
  \item Every set of reals in $L(\uB, \mathbb{R})$ is Lebesgue measurable and has the Baire property and the perfect set property.
 \end{itemize}
\end{prop}
\begin{proof}
 By Lemma \ref{lem:forcing} every universally Baire set of reals in $V[G]$ is definable in a uniform way from $V$, $G$, and elements of $V_\kappa$, namely an ordinal $\alpha < \kappa$ and a $\Col(\omega,\mathord{<}\alpha)$-name for a $\kappa$-universally Baire set of reals.  (Recall that the canonical extension of a $\kappa$-universally Baire set of reals does not depend on any choice of trees.) Because $V_\kappa$ of the ground model has cardinality $\omega_1$ in $V[G]$, it follows that $V[G]$ satisfies $\left|\uB\right| = \omega_1$.
 
 Now let $A$ be a set of reals in $L(\uB, \mathbb{R})^{V[G]}$. Then $A$ is ordinal-definable in $V[G]$ from a universally Baire set and a real. The real is in $V[G \restriction \alpha]$ for some $\alpha < \kappa$ by the inaccessibility of $\kappa$, and the universally Baire set is definable in $V[G]$ from parameters in $V[G \restriction \alpha]$ for some $\alpha < \kappa$ by Lemma \ref{lem:forcing}, so $A$ itself is definable in $V[G]$ from parameters in $V[G \restriction \alpha]$ for some $\alpha < \kappa$. Therefore the proof of Solovay \cite[Theorem 1]{SolEverySetMeasurable} shows that $A$ has the claimed regularity properties.
\end{proof}

Proceeding with the proof of Theorem \ref{thm:v-Shelah}, note that statement \ref{item:equi-L-R-uB-PSP} of the theorem, which says that every set of reals in $L(\uB, \mathbb{R})$ has the perfect set property, obviously implies statement \ref{item:equi-uB-PSP} of the theorem, which says that every universally Baire set of reals has the perfect set property. We will show that statements \ref{item:omega-1-uB-sets} and \ref{item:equi-L-R-uB-LM} also imply statement \ref{item:equi-uB-PSP}.

Assume that statement \ref{item:equi-uB-PSP} fails, so there is an uncountable thin universally Baire set $A$. Then by Lemma \ref{lem:thin-equivalences} every subset of $A$ is also universally Baire, so $\left|\uB\right| \ge 2^{\omega_1}$ and therefore statement \ref{item:omega-1-uB-sets} fails. Also by Lemma \ref{lem:thin-equivalences} there is a universally Baire wellordering of $A$, so $L(\uB, \mathbb{R}) \models \omega_1 \le \left|\mathbb{R}\right|$. Moreover, because every countable sequence of universally Baire sets is coded by a universally Baire set (and similarly for reals) we have $L(\uB, \mathbb{R}) \models \text{DC}$.  Therefore by Shelah \cite[Theorem 5.1B]{SheSolovaysInaccessible} the model $L(\uB, \mathbb{R})$ has a set of reals that is not Lebesgue measurable, so statement \ref{item:equi-L-R-uB-LM} fails.

Next we show that statement \ref{item:equi-uB-PSP} implies statement \ref{item:equi-combinatorial}:

\begin{lem}\label{lem:PSP-uB-implies-combinatorial}
 Assume that every universally Baire set of reals has the perfect set property. Then for every function $f: \omega_1 \to \omega_1$ there is an ordinal $\lambda > \omega_1$ such that for a stationary set of $\sigma \in \powerset_{\omega_1}(\lambda)$ we have $\sigma \cap \omega_1 \in \omega_1$ and $\otp(\sigma) \ge f(\sigma \cap \omega_1)$.
\end{lem}
\begin{proof}
 Assume toward a contradiction that some function $f: \omega_1 \to \omega_1$ fails to have this property. For every ordinal $\lambda > \omega_1$ the set $\{\sigma \in \powerset_{\omega_1}(\lambda) : \sigma \cap \omega_1 \in \omega_1\}$ is always club in $\powerset_{\omega_1}(\lambda)$, so it follows from our assumption that the set of all $\sigma \in \powerset_{\omega_1}(\lambda)$ such that $\sigma \cap \omega_1 \in \omega_1$ and $\otp(\sigma) < f(\sigma \cap \omega_1)$ contains a club set in $\powerset_{\omega_1}(\lambda)$. (Note that this statement trivially holds for $\lambda = \omega_1$ also.) Increasing the values of $f$ if necessary, we may assume that $f$ is a strictly increasing function.
 
 We may consider every real $x$ as coding a structure $(\omega; E_x)$ where $E_x$ is a binary relation on $\omega$.  More precisely, for all $m,n \in \omega$ we let $(m,n) \in E_x$ if and only if $x(\langle m,n \rangle) = 0$ where the angle brackets denote a recursive pairing function. For every countable ordinal $\beta$ we may choose a real $x_\beta$ coding $\beta$ in the sense that $(\omega; E_{x_\beta}) \cong (\beta, \in)$.
 We claim that the set of reals 
 \[ A = \{x_\beta : \beta \in \ran(f)\}, \] 
 which is uncountable, is also universally Baire and thin.  This will contradict our assumption that every universally Baire set of reals has the perfect set property.
 
 Because $\left|A\right| = \omega_1$ we may trivially define a tree $T$ on $\omega \times \omega_1$ from $A$ such that $\p[T] = A$ in every generic extension of $V$. To prove the claim, it will therefore suffice to show that for every poset $\mathbb{P}$ there is a tree $\tilde{T}$ on $\omega \times \Ord$ such that the pair $(T,\tilde{T})$ is $\mathbb{P}$-absolutely complementing, which for this trivial tree $T$ simply means 
 \[ {} \forces_{\mathbb{P}} \p[\tilde{T}] = \omega^\omega \setminus A.\]
 This will show that $A$ is universally Baire by definition, and also that $A$ is thin by condition \ref{item:no-new-elements} of Lemma \ref{lem:thin-equivalences}.  (Our assumption on $f$ is only needed to prove universal Baireness of $A$.  Thinness of $A$ follows from the more general fact that the set $\{x_\beta : \beta <\omega_1\}$ is thin, which can be proved using the boundedness lemma; see Jech \cite[Corollary 25.14]{Jec2002}.)
 
 Fix a poset $\mathbb{P}$ and let $\eta = \max(|\mathbb{P}|^+,\omega_1)$. Using our hypothesis on the existence of club sets, for every ordinal $\lambda$ in the interval $[\omega_1,\eta)$ we may choose a function 
 \[g_\lambda : \lambda^{\mathord{<}\omega} \to \lambda\]
 such that for every $g_\lambda$-closed set $\sigma \in \powerset_{\omega_1}(\lambda)$ we have $\sigma \cap \omega_1 \in \omega_1$ and $\otp(\sigma) < f(\sigma \cap \omega_1)$. (If $\mathbb{P}$ is countable then this interval is empty and there is nothing to do in this step.) Moreover, we can choose $g_\lambda$ to satisfy the additional property that for every $g_\lambda$-closed set $\sigma \in \powerset_{\omega_1}(\lambda)$ the set $\sigma \cap \omega_1$ is $f$-closed. Then because $f$ is strictly increasing, these properties imply that for every $g_\lambda$-closed set $\sigma \in \powerset_{\omega_1}(\lambda)$ the order type of $\sigma$ is not in the range of $f$.
 
 As a first step in defining the tree $\tilde{T}$, note that there is a tree $\tilde{T}_1$ on $\omega \times \omega$ in $V$ such that in every generic extension of $V$ we have 
 \[ \p[\tilde{T}_1] = \{x \in \omega^\omega: (\omega;E_x) \text{ is not a well-ordering}\},\]
 because this set of reals is $\Sigma^1_1$.
 
 As a second step in defining $\tilde{T}$, note that for every ordinal $\beta < \omega_1^V$ there is a tree $\tilde{T}_{2,\beta}$ on $\omega \times \omega$ in $V$ such that in every generic extension of $V$ we have
 \[\p[\tilde{T}_{2,\beta}] = \{x \in \omega^\omega : (\omega;E_x) \cong (\beta;\mathord{\in}) \text{ and } x \notin A\},\]
 because this set of reals is $\Sigma^1_1(x_\beta)$ where $x_\beta$ is our chosen code of $\beta$.
 (The condition $x \notin A$ in the definition of this set either subtracts the single point $x_\beta$ from the set or leaves it unchanged, according to whether or not $\beta$ is in the range of $f$.)
 
 As a third step in defining $\tilde{T}$, note that for every ordinal $\lambda \in [\omega_1^V, \eta)$ there is a tree $\tilde{T}_{3,\lambda}$ on $\omega \times \lambda$ in $V$ such that in every generic extension of $V$ we have
 \[\p[\tilde{T}_{3,\lambda}] = \{x \in \omega^\omega : (\omega;E_x) \cong (\sigma;\mathord{\in}) \text{ for some $g_\lambda$-closed set $\sigma \in \powerset_{\omega_1}(\lambda)$}\}.\]
 To show this, it suffices to represent the set of all such reals $x$ as the projection of a closed subset of $\omega^\omega \times \lambda^\omega$ onto its first coordinate.  Using a definable pairing function $\lambda \times \omega \cong \lambda$, it equivalently suffices to represent the set of all such reals $x$ as the projection of a closed subset of $\omega^\omega \times \lambda^\omega \times \omega^\omega$ onto its first coordinate. An example of such a closed set is the set of all triples $(x, h, y) \in \omega^\omega \times \lambda^\omega \times \omega^\omega$ such that $h$ is an embedding $(\omega; E_x) \to (\lambda;\mathord{\in})$ and $y$ is a function telling us how far we have to look ahead in order to verify that the range of $h$ is $g_\lambda$-closed. To be precise, we require $g_\lambda[ \{ h(i) : i < n \}^{\mathord{<}n} ] \subset \{ h(i) : i < y(n) \}$ for all $n < \omega$.
 
 Now we can define a tree $\tilde{T}$ on $\omega \times \eta$ in $V$ as an amalgamation of these trees, so in every generic extension of $V$ we have
 \[\p[\tilde{T}] = \p[\tilde{T}_1] \cup \bigcup_{\beta < \omega_1^V} \p[\tilde{T}_{2,\beta}] \cup \bigcup_{\lambda \in [\omega_1^V, \eta)} \p[\tilde{T}_{3,\lambda}].\]
 Let $G \subset \mathbb{P}$ be a $V$-generic filter and let $x$ be a real in $V[G]$. We want to show 
 \[x \in A \iff x \notin \p[\tilde{T}].\]
 
 Assume that $x \in A$. Then clearly we have $x \notin \p[\tilde{T}_1]$ and $x \notin \p[\tilde{T}_{2,\beta}]$ for all $\beta < \omega_1^V$. By the definition of $A$ we have $x = x_\beta$ for some $\beta$ in the range of $f$. Now let $\lambda \in [\omega_1^V, \eta)$. As previously noted, the order type of a $g_\lambda$-closed set $\sigma \in \powerset_{\omega_1}(\lambda)^V$ cannot be in the range of $f$, so we have $x \notin \p[\tilde{T}_{3,\lambda}]$ in $V$ and equivalently in $V[G]$. Therefore $x \notin \p[\tilde{T}]$.
 
 Conversely, assume that $x \notin A$. If $(\omega;E_x)$ is not a wellordering then we have $x \in \p[\tilde{T}_1]$. If $(\omega;E_x)$ is a wellordering of order type less than $\omega_1^V$ then let $\beta$ be its order type and note that $x \in \p[\tilde{T}_{2,\beta}]$. If $(\omega;E_x)$ is a wellordering of order type greater than or equal to $\omega_1^V$ then let $\lambda$ be its order type. Note that $\lambda < \eta$ because $\eta$ was chosen to be large enough that forcing with $\mathbb{P}$ does not collapse it to be countable. Therefore $\lambda \in [\omega_1^V,\eta)$ and we have $x \in \p[\tilde{T}_{3,\lambda}]$ because $(\omega;E_x) \cong (\lambda;\mathord{\in})$ and the set $\lambda \in \powerset_{\omega_1}(\lambda)^{V[G]}$ is trivially $g_\lambda$-closed. In every case we have shown that $x \in \p[\tilde{T}]$, as desired.
\end{proof}

Finally, we show that if statement \ref{item:equi-combinatorial} holds then statement \ref{item:equi-v-Shelah} holds in $L$:

\begin{lem}\label{lem:combinatorial-implies-v-Shelah-in-L}
 Assume that for every function $f: \omega_1 \to \omega_1$ there is an ordinal $\lambda > \omega_1$ such that for a stationary set of $\sigma \in \powerset_{\omega_1}(\lambda)$ we have $\sigma \cap \omega_1 \in \omega_1$ and $\otp(\sigma) \ge f(\sigma \cap \omega_1)$. Then $\omega_1^V$ is virtually Shelah in $L$.
\end{lem}
\begin{proof}
 Let $\kappa = \omega_1^V$. First we will show that $\kappa$ is inaccessible in $L$. Because $\kappa$ is regular in $V$ it is regular in $L$, so by GCH in $L$ it remains to show that $\kappa$ is a limit cardinal in $L$. Suppose toward a contradiction that there is a cardinal $\eta < \kappa$ in $L$ such that $(\eta ^+)^L = \kappa$ and define the function $f : \kappa \to \kappa$ in $L$ such that for every ordinal $\alpha < \kappa$, $f(\alpha)$ is the least ordinal $\beta > \max(\alpha, \eta)$ such that $L_\beta \models \left| \alpha \right| \le \eta$. (Note that $\beta < \kappa$ by G\"{o}del's condensation lemma.)

 By our assumption there is an ordinal $\lambda > \kappa$ such that for a stationary set of $\sigma \in \powerset_\kappa(\lambda)$ in $V$ we have $\sigma \cap \kappa \in \kappa$ and $\otp(\sigma) \ge f(\sigma \cap \kappa)$.  Because the set of all $X \in \powerset_\kappa(L_\lambda)$ such that $X \prec L_\lambda$, $X \cap \kappa \in \kappa$, and $\eta \cup \{\eta\} \subset X$ is club in $\powerset_\kappa(L_\lambda)$ there is such a set $X$ with the additional property that $\bar{\lambda} \ge f(\bar{\kappa})$ where $\bar{\kappa} = X \cap \kappa$ and $\bar{\lambda} = \otp(X \cap \lambda)$. Note that
 \[ \eta < \bar{\kappa} < f(\bar{\kappa}) \le \bar{\lambda} < \kappa < \lambda.\]
 By the condensation lemma, $X$ is the range of an elementary embedding
 \[j : L_{\bar{\lambda}} \to L_\lambda \text{ with } \crit(j) = \bar{\kappa} \text{ and } j(\bar{\kappa}) = \kappa.\]
 By the definition of $f$ we have $L_{f(\bar{\kappa})} \models \left| \bar{\kappa} \right| \le \eta$, and because $f(\bar{\kappa}) \le \bar{\lambda}$ it follows that $L_{\bar{\lambda}} \models \left| \bar{\kappa} \right| \le \eta$.  Then we have $L_\lambda \models \left| \kappa \right| \le \eta$ by the elementarity of $j$, contradicting the fact that $\kappa$ is a cardinal in $V$.

 Now because $\kappa$ is inaccessible in $L$ it follows that $L_\beta = V_\beta^L$ for a cofinal set of ordinals $\beta < \kappa$.  We will use this fact to show that $L$ satisfies statement \ref{item:kappa-is-image-of-crit-j} of Proposition \ref{prop:downward-reformulation}, which is an equivalent condition for $\kappa$ to be virtually Shelah.

 Let $f : \kappa \to \kappa$ be a function in $L$ and define the function $g: \kappa \to \kappa$ in $L$ where for every ordinal $\alpha < \kappa$, $g(\alpha)$ is the least ordinal $\beta \ge \max(f(\alpha), \alpha+1)$ such that $L_\beta = V_\beta^L$.  Applying our assumption to the function $g + 1$, we obtain an ordinal $\lambda > \kappa$ such that for a stationary set of $\sigma \in \powerset_\kappa(\lambda)$ in $V$ we have $\sigma \cap \kappa \in \kappa$ and $\otp(\sigma) > g(\sigma \cap \kappa)$. By increasing $\lambda$ if necessary, we may assume that $L_\lambda = V_\lambda^L$.

 Because the set of all $X \in \powerset_\kappa(L_\lambda)$ such that $X \prec L_\lambda$, $X \cap \kappa \in \kappa$, and $f \in X$ is club in $\powerset_\kappa(L_\lambda)$ there is such a set $X$ with the additional property that $\bar{\lambda} > g(\bar{\kappa})$ where $\bar{\kappa} = X \cap \kappa$ and $\bar{\lambda} = \otp(X \cap \lambda)$. By the condensation lemma, $X$ is the range of an elementary embedding
 \[j : L_{\bar{\lambda}} \to L_\lambda \text{ with } \crit(j) = \bar{\kappa} \text{ and } j(\bar{\kappa}) = \kappa.\]
 Define $\bar{\beta} = g(\bar{\kappa})$ and $\beta = j(\bar{\beta})$ and $\bar{f} = j^{-1}(f)$. Note that
 \[ \bar{\kappa} < \bar{\beta} < \bar{\lambda} < \kappa < \beta < \lambda.\]
 We have $L_{\bar{\beta}} = V_{\bar{\beta}}^L$ because $\bar{\beta}$ is in the range of the function $g$. Therefore $L_{\bar{\lambda}} \models L_{\bar{\beta}} = V_{\bar{\beta}}$ and it follows by the elementarity of $j$ that $L_\lambda \models L_\beta = V_\beta$.  Because $L_\lambda = V_\lambda^L$, this implies $L_\beta = V_\beta^L$. For the elementary embedding $j_1 = j \restriction V_{\bar{\beta}}^L$ we have
 \[ j_1: V_{\bar{\beta}}^L \to V_\beta^L \text{ and } \crit(j_1) = \bar{\kappa} \text{ and } j_1(\bar{\kappa}) = \kappa \text{ and } j_1(\bar{f}) = f.\]
 Let $G \subset \Col(\omega,V_{\bar{\beta}}^L)$ be a $V$-generic filter. Then by the absoluteness of elementary embeddability of countable structures there is an elementary embedding $j_2 \in L[G]$ such that
 \[ j_2: V_{\bar{\beta}}^L \to V_\beta^L \text{ and } \crit(j_2) = \bar{\kappa} \text{ and } j_2(\bar{\kappa}) = \kappa \text{ and } j_2(\bar{f}) = f.\]
 Because $f(\bar{\kappa}) \le g(\bar{\kappa}) = \bar{\beta}$, this elementary embedding $j_2$ witnesses statement \ref{item:kappa-is-image-of-crit-j} of Proposition \ref{prop:downward-reformulation} for $\kappa$ in $L$ with respect to the function $f$.
\end{proof}

This completes the proof of Theorem \ref{thm:v-Shelah}. The following question remains open.

\begin{question}\label{question:BP}
 What is the consistency strength of the theory ZFC + ``every set of reals in $L(\mathbb{R},\uB)$ has the Baire property''?
\end{question}

An upper bound for the consistency strength is ZFC + ``there is a virtually Shelah cardinal'' by Proposition \ref{prop:forward-direction}.
We do not know any nontrivial lower bound.

\section{Proof of Theorem \ref{thm:delta-1-2-equals-uB}}\label{section:proof-of-thm-delta-1-2-equals-uB}

When forcing with the Levy collapse over $L$, the virtual Shelah property can be used to limit the complexity of the universally Baire sets of reals in the generic extension:

\begin{prop}\label{prop:uB-subset-Delta-1-2}
 Let $\kappa$ be a virtually Shelah cardinal in $L$ and let $G \subset \Col(\omega,\mathord{<}\kappa)$ be an $L$-generic filter. Then $L[G] \models \uB \subset {\bfDelta}^1_2$.
\end{prop}
\begin{proof}
 Let $A$ be a universally Baire set of reals in $L[G]$. Then by Lemma \ref{lem:forcing} there is an ordinal $\alpha < \kappa$ and a $\kappa$-universally Baire set of reals $A_0$ in $L[G\restriction\alpha]$ such that $A = A_0^{L[G]}$. Increasing $\alpha$ if necessary, we may assume that it is a successor ordinal, so $L[G\restriction\alpha] = L[z]$ for some real $z \in L[G]$. Take a $\Col(\omega,\mathord{<}\kappa)$-absolutely complementing pair of trees $(T, \tilde{T})$ in $L[z]$ such that $\p[T]^{L[z]} = A_0$ and therefore $\p[T]^{L[G]} = A$. Define the sets of reals
 \[B_0 = \p[\tilde{T}]^{L[z]} = \mathbb{R}^{L[z]} \setminus A_0 \text{ and } B = \p[\tilde{T}]^{L[G]} = \mathbb{R}^{L[G]} \setminus A.\]
 
 We claim that for every real $x \in L[G]$ we have $x \in A$ if and only if there is a tree $T' \in L[z]$ such that $x \in \p[T']$ and $\p[T'] \cap B_0 = \emptyset$. To prove the forward direction of the claim, note that if $x \in A$ then we can simply let $T' = T$. To prove the reverse direction of the claim, let $x$ be a real in $L[G]$ and assume that $x \in \p[T']$ for some tree $T' \in L[z]$ such that $\p[T'] \cap B_0 = \emptyset$.  Because $\p[T'] \cap B_0 = \emptyset$ and $B_0 = \p[\tilde{T}]^{L[z]}$ we have $\p[T'] \cap \p[\tilde{T}] = \emptyset$ in $L[z]$. The statement $\p[T'] \cap \p[\tilde{T}] = \emptyset$ is absolute between $L[z]$ and $L[G]$ by the absoluteness of wellfoundedness of the tree of all triples $(r, s_1, s_2)$ such that $(r,s_1) \in T'$ and $(r,s_2) \in \tilde{T}$, so because $x \in \p[T']$ we have $x \notin \p[\tilde{T}]$ and therefore $x \in \p[T] = A$.
 
 We can use the claim to define $A$ in $L[G]$ without reference to $T$. Because the class $L[z]$ is $\Sigma_1(z)$
 and the statement $y \notin \p[T']$ (where $y$ is a real) is witnessed by the existence of a rank function for the section tree $T_y$, it follows from the claim that $A$ is $\Sigma_1(z,B_0)$. Because $\text{HC} \prec_{\Sigma_1} V$, this implies that $A$ is $\Sigma_1^{\text{HC}}(z,B_0)$ and is therefore $\Sigma^1_2(z,b)$ where $b$ is a real coding $B_0$. A symmetric argument shows that $B$ is ${\bfSigma}^1_2$ in $L[G]$, so $A$ and $B$ are both ${\bfDelta}^1_2$ in $L[G]$.
\end{proof}

Proposition \ref{prop:uB-subset-Delta-1-2} is not useful as a relative consistency result because the theory ZFC + $(\uB \subset {\bfDelta}^1_2)$ is equiconsistent with ZFC by the theorem of Larson and Shelah mentioned in the introduction. However, we can combine Proposition \ref{prop:uB-subset-Delta-1-2} with the theorem of Feng, Magidor, and Woodin mentioned in the introduction to obtain an equiconsistency result at the level of a $\Sigma_2$-reflecting virtually Shelah cardinal:

\begin{proof}[Proof of Theorem \ref{thm:delta-1-2-equals-uB}]
 Assume there is a $\Sigma_2$-reflecting virtually Shelah cardinal $\kappa$. Then $\kappa$ is clearly $\Sigma_2$-reflecting in $L$ and it is virtually Shelah in $L$ by Lemma \ref{lem:absolute-to-L}. Let $G \subset \Col(\omega,\mathord{<}\kappa)$ be an $L$-generic filter. Because $\kappa$ is  $\Sigma_2$-reflecting in $L$ we have ${\bfDelta}^1_2 \subset \uB$ in $L[G]$ by the proof of Feng, Magidor and Woodin \cite[Theorem 3.3]{FenMagWoo}. On the other hand, because $\kappa$ is virtually Shelah in $L$ we have $\uB \subset {\bfDelta}^1_2$ in $L[G]$ by Proposition \ref{prop:uB-subset-Delta-1-2}. 

 Conversely, assume  $\uB = {\bfDelta}^1_2$. Because ${\bfDelta}^1_2 \subset \uB$, the cardinal $\omega_1^V$ is $\Sigma_2$-reflecting in $L$ by the proof of Feng, Magidor, and Woodin \cite[Theorem 3.3]{FenMagWoo}. Also because  ${\bfDelta}^1_2 \subset \uB$, every ${\bfSigma}^1_2$ set of reals has the perfect set property by Feng, Magidor, and Woodin \cite[Theorem 2.4]{FenMagWoo}. Combining this with the assumption that $\uB \subset {\bfDelta}^1_2$ shows that every $\uB$ set has the perfect set property, so  $\omega_1^V$ is virtually Shelah in $L$ by Lemmas \ref{lem:PSP-uB-implies-combinatorial} and \ref{lem:combinatorial-implies-v-Shelah-in-L}.
\end{proof}

\section{Acknowledgments}

We thank Stamatis Dimopoulos for a conversation with the second author that led to the current form of Definition \ref{defn:kappa-is-crit-j} in which $\kappa$ is $\crit(j)$ rather than $j(\crit(j))$, making the resemblance to Shelah cardinals more apparent.


\end{document}